\documentclass[12pt]{amsart}
\usepackage[dvips]{graphicx}
\usepackage{amsmath,graphics}
\usepackage{amsfonts,amssymb,hyperref}
\usepackage{xypic,accents}
\usepackage{comment}
\specialcomment{proofc}{}{}
\includecomment{proofc}
\theoremstyle{plain}
\newtheorem*{theorem*}{Theorem}
\newtheorem*{lemma*} {Lemma}
\newtheorem*{corollary*} {Corollary}
\newtheorem*{proposition*}{Proposition}
\newtheorem*{conjecture*}{Conjecture}
\newtheorem{theorem}{Theorem}[section]
\newtheorem{lemma}[theorem]{Lemma}
\newtheorem*{theorem1*}{Theorem 1}
\newtheorem*{theorem2*}{Theorem 2}
\newtheorem*{theorem3*}{Theorem 3}

\newtheorem{proposition}[theorem]{Proposition}

\newtheorem{problem}[theorem]{Problem}

\newcommand{\co}{\colon \thinspace}

\theoremstyle{remark}

\newtheorem*{definition}{Definition}

\newtheorem{example*}{Example}
\newtheorem*{claim}{Claim}

\theoremstyle{definition}

\def\op{\operatorname}

\def\G{\Gamma}

   \def\Z{\Bbb{Z}}  
    
 \def\a{\alpha}  
 \def\bp{\begin{pmatrix}}
\def\sm{\setminus} \def\ep{\end{pmatrix}} \def\bn{\begin{enumerate}} 
   \def\en{\end{enumerate}}
\def\ba{\begin{array}} \def\ea{\end{array}}  
   \def\a{\alpha} \def\b{\beta} \def\wti{\widetilde}

\def\ker{\op{ker}}\def\be{\begin{equation}} \def\ee{\end{equation}} 
   
 \def\hom{\op{Hom}}

\def\GG{\mathcal{G}}

\def\ol{\overline}

\def\wti{\widetilde}

\def\what{\widehat}

\def\F{\Bbb{F}}

\begin{document}
\title{Grothendieck rigidity of 3-manifold groups}
\author{Michel Boileau}
\address{
Aix-Marseille Univ, CNRS, Centrale Marseille,
13453 Marseille, France}
\email{michel.boileau@cmi.univ-mrs.fr }

\author{Stefan Friedl}
\address{Fakult\"at f\"ur Mathematik\\ Universit\"at Regensburg\\   Germany}
\email{sfriedl@gmail.com}

\begin{abstract}
We show that fundamental groups of compact, orientable, irreducible 3-manifolds with toroidal boundary are Grothendieck rigid.
\end{abstract}

\maketitle

\section{Introduction}
Given a group $\pi$ we denote by $\what{\pi}$ its profinite completion. (We refer to Section~\ref{section:definition-profinite} for the definition.) In 1970 Grothendieck~\cite[p.~384]{Gr70} implicitly
posed the following problem:

\begin{problem}\emph{\textbf{(Grothendieck)}}\label{qu:grothendieck} Let $\pi$ be a finitely presented residually finite group and let $\Gamma\subset \pi$  be a finitely presented proper subgroup. Is it 
possible for the inclusion induced map $\what{\Gamma}\to \what{\pi}$ of profinite completions to be an isomorphism?
\end{problem}

Bridson--Grunewald~\cite{BG04}  solved Grothendieck's problem by producing many examples of such pairs $(\pi, \Gamma)$ of finitely presented groups such that the inclusion induced map $\what{\Gamma}\to \what{\pi}$ is indeed an isomorphism. An earlier example with a finitely generated (but not finitely presented) proper subgroup 
$\Gamma$ had been given by Platonov--Tavgen~\cite{PT86}.  Grothendieck's problem motivated the following definition introduced by Long-Reid~\cite{LR11}:

\begin{definition} A group $\pi$ is 
\emph{Grothendieck rigid} if for every finitely generated proper subgroup $\Gamma\subset \pi$ the inclusion induced map $\what{\Gamma}\to \what{\pi}$ of profinite completions is \emph{not} an isomorphism. 
\end{definition}

Note that in the definition of Grothendieck rigid one considers all finitely generated subgroups whereas in Grothendieck's question only finitely presented subgroups were considered.  Recall that a group  $\pi$ is called coherent if every finitely generated subgroup is in fact finitely presented.
For example 3-manifold groups are coherent by Scott's core theorem~\cite{Sco73},

It is an interesting question to determine which finitely presented groups are Grothen\-dieck rigid.
Cavendish~\cite{Ca12}, see also  Reid~\cite[Theorem~8.3]{Rei15}, showed that the fundamental groups of  irreducible \emph{closed} 3-manifolds are Grothendieck rigid. We extend this result to irreducible 3-manifolds with toroidal boundary. More precisely, we have the following theorem.

\begin{theorem}\label{thm:grothendieck-rigid}
The fundamental group of any irreducible, orientable, compact, connected 3-manifold with empty or toroidal boundary is Grothendieck rigid.
\end{theorem}

At least according to our proof it seems like the case of non-empty boundary is significantly harder than the closed case.

\subsection*{Conventions}
Unless we say something else, all 3-manifolds in this paper are understood to be connected, compact and orientable.

\subsection*{Acknowledgments}
Work on this paper was supported by the SFB 1085 `Higher Invariants' at the Universit\"at Regensburg funded by the Deutsche Forschungsgemeinschaft (DFG). 
Both authors are very grateful to the Newton Institute at Cambridge where both spent productive and happy weeks in the spring of 2017. We are very grateful to Gareth Wilkes  and Henry Wilton for many very helpful conversations and we are grateful to Martin Bridson for providing feedback on the first version of the paper.

\section{Preliminaries}
\label{section:profinite}

\subsection{Clean manifolds}\label{section:clean}

\begin{definition}
We say that a 3-manifold $W$ is \emph{clean} if the following conditions are satisfied:
\bn
\item $W$ is irreducible, 
\item  $W$ has empty or toroidal boundary, 
\item all JSJ-components of $W$ that are Seifert fibered spaces are products and
\item  no JSJ-component of $W$ is of  the form $T^2\times I$.
\en
\end{definition}

It is a consequence of the characteristic pair theorem, see e.g.\  \cite[p.~138]{JS79} or \cite[Theorem 1.7.7]{AFW15} that a clean 3-manifold does not contain a twisted $I$-bundle over the Klein bottle.

\begin{proposition}\label{prop:covered-by-clean}
Any irreducible 3-manifold with empty or toroidal boundary that is not a closed Seifert fibered manifold, nor finitely covered by a torus-bundle admits a finite covering that is clean.
\end{proposition}

\begin{proof}
Let $N$ be an irreducible 3-manifold with empty or toroidal boundary that is not a closed Seifert fibered manifold. By   \cite{He87} (see also \cite[Section~4.3]{AF13}) there exists a finite regular covering $\wti{N}$  such that all Seifert fibered JSJ-components of $\wti{N}$ are $S^1$-bundles over an orientable surface. Since $N$ is not a closed Seifert fibered manifold the JSJ-components of $\wti{N}$ that are Seifert fibered have boundary, i.e.\ they are products of $S^1$ with an orientable surface. Since $N$ is not finitely covered by a torus bundle, no JSJ-component is of the form $T^2\times I$. 
\end{proof}


\subsection{Girth of the JSJ-graph}

The JSJ-decomposition of an aspherical  3-manifold $M$ with empty or toroidal boundary gives rise to a JSJ-graph 
$\Gamma(M)$ where the vertices correspond to the  JSJ-components and the edges correspond to the JSJ-tori of $M$. There exists a canonical map $\pi\colon M\to \Gamma(M)$ that is given by sending the JSJ-components to the corresponding vertices and by sending for each JSJ-torus $T$ the points in $T\times [-1,1]$ onto the corresponding edge. This map induces an epimorphism 
$\pi_{*}\colon \pi_1(M) \to \pi_1(\Gamma(M))$. 

\begin{definition} The length of a cycle in a graph is the number of edges of the cyle. The \emph{girth} of a graph is the minimal length among the cycles contained in the graph. If the graph does not contain any cycles, its girth is defined to be infinity.
\end{definition}

\begin{proposition}\label{prop:JSJ-girth}
Any clean 3-manifold with empty or toroidal boundary and non-trivial JSJ-decomposition admits a finite covering with JSJ-graph of girth $\geq 3$.
\end{proposition}

Since the map $\pi\colon M\to \Gamma(M)$ induces an epimorphism $\pi_{*}\colon \pi_1(M) \to \pi_1(\Gamma(M))$, any finite cover $\widetilde{\Gamma}$ of the graph $\Gamma(M)$ induces a finite cover 
$\widetilde{M}$ of $M$. Since $M$ is clean it follows from \cite[Theorem~1.9.3]{AFW15} that  the JSJ-decomposition of $M$ lifts to the JSJ-decomposition of $\widetilde{M}$ and thus $\Gamma(\widetilde{M}) = \widetilde{\Gamma}$. 
Therefore Proposition~\ref{prop:JSJ-girth} follows from the following lemma: 

\begin{lemma}\label{lem: girth}
Any finite graph admits a finite cover with girth $\geq 3$.
\end{lemma}

\begin{proof}
 Let $\Gamma$ be a finite graph. Let assume that the set $\mathcal{C}$ of cycles in $\Gamma$ is non empty. Let $C \subset H_1(\Gamma; \mathbb{Z})$ be 
 the finite  set of classes corresponding to the cycles in $\mathcal{C}$. (In our usage of the word ``cycle'' we mean that cycles are injective, hence the number of cycles is indeed finite.) There is a cohomology class $\phi \in H^1(M;\Z)=\hom(\pi_1(M),\Z)$ which does not vanish on any of  the elements in $C$, since they are only finitely many. Choose a prime number $p \geq 3$ which does not divide $\Pi_{ c \in C} \phi(c)$. Let $\phi_p\colon \pi_1(\Gamma) \to \mathbb{Z} \to \mathbb{Z}_p$ be the epimorphism given by the composition of $\phi$ with the projection map $\Z\to \Z_p$ and denote by $\widetilde \Gamma$ the corresponding $p$-sheeted cover. The choice of $p$ implies that the pullback of each cycle in  $\mathcal{C}$ has length at least $p \geq 3$. It follows that $\widetilde \Gamma$ has girth $\geq 3$. 
\end{proof}

\subsection{The profinite completion of a group}\label{section:definition-profinite}
In this section we recall several basic properties of profinite completions. Throughout this section we refer to \cite{RZ10} for details.
Given a group $\pi$  we consider the inverse system
$\{\pi/\G\}_{\G}$ where $\G$ runs over all finite-index normal subgroups of $\pi$. 
The  profinite completion $\what{\pi}$ of $\pi$ is then defined as the inverse limit
of this system, i.e.\
\[ \what{\pi}=\underset{\longleftarrow}{\lim} \,\pi/\G. \]
Note that the natural map $\pi\to \what{\pi}$ is injective if and only if  $\pi$ is residually finite. It follows from  \cite{He87} and the proof of the Geometrization Conjecture that  fundamental groups of 3-manifolds are residually finite. We also point out that any group homomorphism $f\colon A\to B$ induces a homomorphism of profinite completions $\what{f}\colon \what{A}\to \what{B}$. 

The following lemma is a  consequence of a deep result of Nikolov and Segal~\cite{NS07}. We refer to \cite[Proposition~3.2.2 (a) and Theorem~4.2.2]{RZ10} for details.

\begin{lemma}\label{lem:samefinitequotients-profinite}
Let $\pi$ be a finitely generated group. Then for every finite group $G$ the map $\pi\to \what{\pi}$ induces a bijection
$\hom(\what{\pi},G)\to \hom(\pi,G)$. 
\end{lemma} 

For the most part we will be interested in the following corollary.

\begin{lemma}\label{lem:samefinitequotients}
Let $f\colon A\to B$ be a group homomorphism of finitely generated groups that induces an isomorphism of profinite completions. Then for any finite group $G$ the map
\[ \ba{rcl} \{\mbox{epimorphisms from $B$ onto $G$}\}&\to &\{\mbox{epimorphisms from $A$ onto $G$}\}\\
\beta&\mapsto & \beta\circ f\ea\]
is a bijection.
\end{lemma} 

The following lemma follows from  \cite[Proposition~3.2.2 (d) and Theorem~4.2.2]{RZ10}.

\begin{lemma}\label{lem:same-finite-index-subgroups}
If $f\colon A\to B$ is a homomorphism between  two finitely generated  groups 
that induces an isomorphism of  profinite completions, then for any finite-index subgroup $\Gamma$  of $B$ $($not necessarily normal$)$ the map
\[ A/f^{-1}(\Gamma)\,\,\to\,\,B/\Gamma\]
is a bijection.
\end{lemma} 

\begin{lemma}\label{lem:sameh1}
If $f\colon A\to B$ is a homomorphism between  two finitely generated residually finite groups 
that induces an isomorphism of  profinite completions, then the induced map on homology $f_*\colon H_1(A;\Z) \to H_1(B;\Z)$ is an isomorphism.
\end{lemma} 

\begin{proof}
It follows from Lemma~\ref{lem:samefinitequotients} that the map $f$ induces an epimorphism from $A$ onto every finite abelian quotient of $B$. It follows easily from the classification of finitely generated abelian groups that $f_*\colon H_1(A;\Z)\to H_1(B;\Z)$ is an epimorphism.
\end{proof}

\begin{lemma}\label{lem:virtiso}
Let $f\co A\to B$ be a group homomorphism of finitely generated groups and let  $\b\colon B\to G$ be an epimorphism onto a finite group. If $\what{f}\colon \what{A}\to \what{B}$ is an isomorphism, then $\what{f}$ restricts to an isomorphism
\[ \what{\ker(\b\circ f)}\to \what{\ker(\b)}.\]
\end{lemma}

\begin{proof}
First note that if $C$ is a finitely generated groups and if  
$\gamma\colon C\to G$ is an epimorphism onto a finite group, then $ \what{\ker(\gamma)}=\ker(\gamma\colon \what{C}\to G)$.
Now let $\b\colon B\to G$ be an epimorphism onto a finite group. 
We denote by $\beta\colon \what{B}\to G$ the corresponding homomorphism.

We write $\a=\b\circ f\colon A\to G$. 
It follows from Lemma~\ref{lem:samefinitequotients} that $\alpha\colon A\to G$ is  an epimorphism and that the induced map $\alpha\colon \what{A}\to G$ is also an epimorphism.
We  obtain the following commutative diagram
\[ \xymatrix@R0.5cm{ 1\ar[r] & \ker(\a\colon \what{A}\to G)\ar[d]^-f\ar[r]&\what{A}\ar[d]^-f\ar[r]^\alpha& G\ar[d]\ar[r]& 1\\
1\ar[r] & \ker(\a\colon \what{B}\to G)\ar[r]&\what{B}\ar[r]^\beta& G\ar[r]& 1.}\]
Since the map in the middle is an isomorphism and the map on the right is an isomorphism it follows that the map on the left is also an isomorphism. The lemma follows from the fact, mentioned at the beginning of the proof, that the groups on the left agree with the profinite completions of the kernels of $\a\colon A\to G$ and $\b\colon B\to G$.
\end{proof} 

We conclude this section with the following elementary lemma.

\begin{lemma}\label{lem:virtiso-2}
Let $f\co A\to B$ be a group homomorphism and let  $\b\colon B\to G$ be an epimorphism onto a finite group. Then the following are equivalent:
\bn
\item $f$ is an isomorphism,
\item $ \beta\circ f\colon A\to G$ is an epimorphism and the restriction of $f$ to $\ker(\beta\circ f)\to \ker(\beta)$ is an isomorphism.
\en
\end{lemma}

%

\subsection{Graphs of groups and the JSJ-decomposition of a 3-manifold}
A graph of groups $\GG$ consists of an oriented graph $\Gamma(\GG)=(V(\GG),E(\GG))$ together with the following data:
\bn
\item for each vertex $v\in V(\GG)$  we have a group $\pi_v$,
\item for each edge $e\in V(\GG)$  we have a group $\pi_e$ such that $\pi_e=\pi_{\ol{e}}$, where $\ol{e}$ is the edge $e$ with the opposite orientation.
\item for each edge $e$ we have group monomorphisms $\alpha_e\colon \pi_e\to \pi_{i(e)}$ and  $\beta_e\colon \pi_e\to \pi_{t(e)}$ with $\alpha_e=\beta_{\ol{e}}$ and 
 $\alpha_{\ol{e}}=\beta_{{e}}$, where $i(e)$ and $t(e)$ are the initial and terminal vertices of the oriented edge $e$.
\en
Given a graph $G$, by a slight abuse of notation we also denote by $G$ its topological realization.
Given a graph of groups $\GG$ we denote by $\pi_1(\GG)$ its fundamental group.
We refer to \cite{Se80} for precise definitions and for more information.

A morphism $f$ between two graphs of groups $\GG$ and $\GG'$ consists of a map 
\[ f\colon (V(\GG),E(\GG))\,\,\to\,\, (V(\GG'),E(\GG'))\]
 of the graphs and homomorphisms $\pi_{v}\to \pi_{f(v)}$, $v\in V(\GG)$ and $\pi_{e}\to \pi_{f(e)}$, $e\in V(\GG)$ which respect the maps $\alpha_e$ and $\beta_e$.
Such a morphism induces a homomorphism $f_*\colon \pi_1(\GG)\to \pi_1(\GG')$.

The following elementary lemma gives a criterion for $f_*\colon \pi_1(\GG)\to \pi_1(\GG')$ to be an isomorphism.

\begin{lemma}\label{lem:epimorphism}
Let $f\colon \GG\to \GG'$ be a morphism between two graphs of groups with the following properties:
\bn
\item $f\colon \Gamma(\GG)\to \Gamma(\GG')$ is an isomorphism of graphs,
\item for each $v\in V(\GG)$ the map $\pi_1(\pi_v)\to \pi_1(\pi_{f(v)})$ is an isomorphism,
\item for each $e\in E(\GG)$ the map $\pi_1(\pi_e)\to \pi_1(\pi_{f(e)})$ is an isomorphism.
\en
Then  $f_*\colon \pi_1(\GG)\to \pi_1(\GG')$ is an isomorphism.
\end{lemma}

The main example of a graph of groups that we have in mind is the following. Given an aspherical  3-manifold $M$ with empty or toroidal boundary we obtain a JSJ-decomposition, which gives rise to the JSJ-graph 
$\Gamma(M)$. Recall that  the vertices correspond to the  JSJ-components and the edges correspond to the JSJ-tori of $M$.  Furthermore recall that there exists a canonical map $M\to \Gamma(M)$ that is given by sending the JSJ-components to the corresponding vertices and by sending for each JSJ-torus $T$ the points in $T\times [-1,1]$ onto the corresponding edge.

The manifold $M$ also gives rise to a graph of groups $\GG(M)$ such that $\pi_1(\GG(M)) \cong \pi_1(M)$, where the vertex groups are the fundamental groups of the JSJ-components and the edge groups are the fundamental groups of the JSJ-tori, and the monomorphisms from edge groups to vertex groups are the inclusion induced maps. Let $\mathfrak{T}(M)$ denote the associated Bass-Serre tree, then $\Gamma(M) = \Gamma(\GG(M)) = \mathfrak{T}(M)/\pi_1(M)$ is its quotient by the action of $\pi_1(M)$.

\subsection{Efficiency of the JSJ-graph of groups}

We recall a definition from group theory.

\begin{definition}
Let $\pi$ be a group. We say that a subgroup $\Gamma$ is \emph{separable} if for any $g\in \pi\sm \Gamma$ there exists a
homomorphism $\alpha\colon \pi\to G$ to a finite group $G$ such that $\alpha(g)\not\in \alpha(\Gamma)$.
\end{definition}

The following theorem is \cite[Theorem~A]{WZ10}.

\begin{theorem}\label{thm:wz10}
Let $M$ be an aspherical 3-manifold. Then the corresponding JSJ-graph of groups $\GG(M)$ is ``efficient'', i.e.\  the following holds:
\bn
\item If $M_v, v\in V$ denote the JSJ-components of $M$, then given any choice of finite-index-subgroups $H_v \subset \pi_1(M_v)$, $v \in V$ there exists a finite-index subgroup $H$ of $\pi_1(M)$, such that $H \cap \pi_1(M_v)\subset H_v$ for all $v\in V$.
\item The fundamental groups of the JSJ-tori and JSJ-components are separable in $\pi_1(M)$. 
\en
\end{theorem}

In the setting of finite graphs of profinite groups there is an analogue of the classical Bass-Serre theory. A profinite graph is an inverse limit of finite graphs. It is connected if each of its finite quotients is connected. 
Its profinite fundamental group is the group of deck transformations of its profinite universal cover, defined by satisfying an appropriate universal property among its connected Galois covers, see \cite[Chapter 2]{Rib17} for details.
 A profinite graph is called a \emph{profinite tree} if its profinite fundamental group is trivial. The definition of the fundamental group of a finite graph of profinite groups is analogous to the one in the classical Bass-Serre theory. One can associate to a finite graph of profinite groups $\what{\GG}$ a profinite Bass-Serre tree $\what{\mathfrak{T}}$ on which the fundamental group $\pi_1 \what{\GG}$ acts. Similarly to the classical case, the profinite Bass-Serre tree is built as a set of  right cosets of the edge and vertex groups in the profinite fundamental group of the profinite graph of groups,
see \cite[Section 3]{ZM89}, \cite{ZM90}, \cite [Chapter 6]{Rib17}. Background material and details for the theory of profinite groups acting on profinite trees can be found also in \cite{RZ96}, 
\cite[Sections 1,2]{RZ00}.

The next result is a consequence of the efficiency of the JSJ-graph of groups of a 3-manifold, together with \cite[Section 3]{ZM89}, \cite[Proposition 3.2]{RZ96}, \cite[Proposition 2.5]{C-B13}, \cite[Section 6.3]{Rib17} (see also \cite[Theorems 5.4 and 5.6]{Wil17a}, \cite[Section 1.1]{WZ17b})

\begin{theorem}\label{thm:profinitegraph}
Let $M$ be an aspherical  3-manifold whose  JSJ-graph of groups is $\GG(M)$ and associated Bass-Serre tree $\mathfrak{T}(M)$.
 Let $\what{\GG}(M)$ be the finite graph of profinite groups obtained from $\GG(M)$ by replacing each vertex and edge group by its profinite completion and each monomorphism from an edge group into a vertex group by its extension to the profinite completions. Then the following hold:
\begin{enumerate}
\item[(a)]  Each profinite edge group of $\what{\GG}(M)$ injects in the profinite fundamental group $\pi_1(\what{\GG}(M))$ and 
 $\what{\pi_1(\GG(M))} \cong \pi_1(\what{\GG}(M))$. 
\item[(b)]  There is an associated profinite Bass-Serre tree $\what{\mathfrak{T}}(M) = \underset{\longleftarrow}{\lim} \,\mathfrak{T}(M) /G$, where $G$ runs over all normal finite index subgroups of $\pi_1(M)$ such that the following hold:
\begin{enumerate}
\item[(b1)] The Bass-Serre tree $\mathfrak{T}(M)$ associated to the JSJ-decomposition of $M$ embeds and is dense in  $\what{\mathfrak{T}}(M)$.\
\item[(b2)] The profinite group $\what{\pi_1(M)}$ acts continuously on $\what{\mathfrak{T}}(M)$ and
 \[ \what{\mathfrak{T}}(M)/\what{\pi_1(M)}\, =\, \Gamma(M)\, =\, \mathfrak{T}(M)/\pi_1(M).\]
\item[(b3)]  If $\what{p}_M\colon \what{\mathfrak{T}}(M) \to \what{\mathfrak{T}}(M)/\what{\pi_1(M)} = \Gamma(M)$ is the projection, the stabilizer of a vertex $\wti{v} \in V(\what{\mathfrak{T}}(M))$ is conjugated to the profinite completion  of the vertex group associated to the vertex $\what{p}_M(\wti{v}) \in V(\Gamma(M))$ in the  JSJ-graph of groups $\GG(M)$. 
\end{enumerate}
\end{enumerate}
\end{theorem}

\subsection{Acylindrical actions}
The following definition is taken from \cite[Definition~3.2]{WZ17b}, see also \cite[Section~4]{Wil17a}.

\begin{definition}
An action of a profinite group $\what{\pi}$ on a 
profinite tree $\what{\mathcal{T}}$ is called \emph{$k$-acylindrical} if for every $g\in \what{\pi}$ the subtree fixed by $g$ is either empty or
of diameter at most $k$. Such an action is called \emph{acylindrical}
if it is $k$-acylindrical for some $k$.
\end{definition}

In the following one can treat the definition of an acylindrical action as a black box.
What is relevant to us is summarized in the following two lemmas.

\begin{lemma}\label{lem:action-is-acylindrical}
Let $M$ be a clean 3-manifold  with non-empty boundary. Let $\mathcal{G}(M)$ be the graph of groups defined by the JSJ-decomposition of $M$. Then the action of $\what{\pi_1(M)}$ on the profinite tree $\what{\mathcal{T}}(M)$ is 2-acylindrical.
\end{lemma}

\begin{proof}
Let $M$ be a  clean 3-manifold  with non-empty boundary. 
Denote by $N$ the result of gluing hyperbolic 3-manifolds to $M$ along the boundary components of $M$.
Note that $N$ is again an orientable 3-manifold. By \cite[Proposition~1.6.2]{AFW15} the JSJ-decomposition of $N$ is given by the JSJ-decomposition of $M$ together with the hyperbolic manifolds we had just attached.

Since $M$ is clean, every Seifert fibered JSJ-component of  $N$ admits an epimorphism onto a free group, in particular it is large in the sense of \cite{HWZ13}. Then it follows from \cite[Lemma~4.11]{HWZ13} that the action of $\pi_1(N)$ on $\mathcal{T}(N)$ is 2-acylindrical. Since $\mathcal{T}(M)\subset \mathcal{T}(N)$ this also implies that the action of $\pi_1(M)$ on $\mathcal{T}(M)$ is 2-acylindrical.
\end{proof}

The following lemma is a slight rewriting of \cite[Lemma 4.4 and 4.5]{WZ17b}, see also \cite[Proposition 6.23]{Wil17b}

\begin{lemma}\label{lem:fixed-vertex}
\bn
\item Consider $N$ a compact, hyperbolic 3-manifold with $($possibly
empty$)$ toroidal boundary. If $\what{\pi_1(N)}$ acts acylindrically on a profinite tree $\what{\mathcal{T}}$ with
abelian edge stabilizers, then $\what{\pi_1(N)}$  fixes a unique vertex.
\item Let $N=S^1\times \Sigma$ where $\Sigma$ is a surface with $\chi(\Sigma)<0$.  If $\what{\pi_1(N)}$ acts acylindrically on a profinite tree
$\what{\mathcal{T}}$ with abelian edge stabilizers, then $\what{\pi_1(N)}$ fixes a unique vertex.
\en
\end{lemma}





\subsection{Good groups}
In this short section  we recall that, following Serre~\cite[D.2.6.~Exercise D]{Se97} a group $\pi$ is called \emph{good} if for every finite abelian group $A$ and any representation $\alpha\colon \pi\to \op{Aut}(A)$ and any $i$ the natural map
\[ H^i(\what{\pi};A) \to H^i(\pi;A)\]
is an isomorphism.

The following theorem is due to the  work of Wise~\cite{Wis09,Wis12a,Wis12b}, to works of Wilton--Zalesskii~\cite{WZ10} and Cavendish~\cite[Section~3.5~and Lemma~3.7.1]{Ca12}. We 
also refer to  \cite{CF17,GM17}for alternative proofs and we refer to \cite[(H.26)]{AFW15} for details.

\begin{theorem}\label{thm:good}
The fundamental group of every 3-manifold $($with no restrictions on the boundary$)$ is good.
\end{theorem}

\subsection{Monomorphisms of 3-manifold groups}

\begin{definition}
Let $M$ be an irreducible 3-manifold with empty or toroidal boundary.
The \emph{characteristic submanifold $\Sigma M$} of $M$ is the union of the following:
\bn
\item all JSJ-components of $M$ that are Seifert fibered,
\item all boundary tori of $M$ which cobound an atoroidal JSJ-component,
\item all JSJ-tori of $M$.
\en
\end{definition}

By \cite[Theorem~1.7.7]{AFW15} the above characteristic submanifold is precisely the same object as introduced by Jaco-Shalen~\cite[Chapter~V]{JS79}.



We say that a map $f\colon M\to N$ between two irreducible 3-manifolds with empty or toroidal boundary is \emph{nice} if $\Sigma M$ is a collection of components of $f^{-1}(\Sigma N)$. The following lemma follows immediately from the definition of the JSJ-graph and the definition of a nice map.

\begin{lemma}\label{lem:induced-graph map}
A nice  map $f\colon M\to N$ between two irreducible 3-manifolds with empty or toroidal boundary
induces a map $f\colon \Gamma(M)\to \Gamma(N)$ such that the following diagram commutes
\[ \xymatrix@C1.2cm@R0.5cm{ M\ar[d]\ar[r]^{f} & N \ar[d] \\ \Gamma(M)\ar[r]_{f}&\Gamma(N).}\]
\end{lemma}

\begin{definition}
Let $N$ be an irreducible 3-manifold with empty or toroidal boundary.
We denote by $\beta N$ the union of all Seifert pieces of $\Sigma$ that contain boundary components of $N$.
\end{definition}

The following deformation theorem is a combination of \cite[Proposition~1.2]{GAW92}  and \cite[Deformation Theorem~1.2]{GAW92} which in turn is a consequence of work of Johannson~\cite[p.~127-128]{Jo79} and Waldhausen~\cite[Theorem~6.1]{Wa68}.

\begin{theorem}\label{thm:deformationtheorem}
Let $K$ and $N$ be two irreducible 3-manifolds with non-empty toroidal boundary that do not contain $\pi_1$-injective Klein bottles.
Let $g\colon K\to N$ be a $\pi_1$-injective map. The map $g$ can be homotoped to a map $f$ satisfying the following conditions:
\bn
\item $f$ is nice,
\item if $C$ is a component of either $f^{-1}(\Sigma N)\sm \beta K$ or of  $\ol{K\sm f^{-1}(\Sigma N)}$, then the map $f\colon C\to f(C)$ is a covering.
\en
\end{theorem}

\section{Proof of Theorem~\ref{thm:grothendieck-rigid}}\label{section:grothendieck-rigid}
For the reader's convenience we recall the statement of Theorem~\ref{thm:grothendieck-rigid}.\\

\noindent 
\textbf{Theorem~\ref{thm:grothendieck-rigid}.} \emph{If $M$ is an irreducible, orientable, compact, connected 3-manifold with empty or toroidal boundary, then $\pi_1(M)$ is Grothendieck rigid.}\\

\subsection{Monomorphisms of 3-manifold groups and profinite completions}

\begin{proposition}\label{prop:2}
Let $f\colon K\to M$ be a nice map between two irreducible  3-manifolds with empty or toroidal boundary such that $\what{f}_*\colon \what{\pi_1(K)}\to \what{\pi_1(M)}$ is an isomorphism. Then the induced map $f_*\colon \pi_1(\Gamma(K))\to \pi_1(\Gamma(M))$, that is defined by 
Lemma~\ref{lem:induced-graph map},  is an epimorphism.
\end{proposition}

\begin{proof}
By Lemma~\ref{lem:induced-graph map} we have the following commutative diagram
\[ \xymatrix@R0.6cm{ \pi_1(K)\ar[d] \ar[r]^{f_*} & \pi_1(M)\ar@{->>}[d]^p \\
\pi_1(\Gamma(K))\ar[r] & \pi_1(\Gamma(M)).}\]
The group $\pi_1(\Gamma(M))$ is a free group.
It follows from \cite[Theorem 5.1]{Ha49} that every finitely generated subgroup of $\pi_1(\Gamma(M))$ is separable.
Therefore, if the induced map $f_*\colon \pi_1(\Gamma(K))\to  \pi_1(\Gamma(M))$ is not an epimorphism, then there exists a finite-index proper subgroup $\Gamma$ of $ \pi_1(\Gamma(M))$ with 
$f_*(\pi_1(\Gamma(K)))\subset \Gamma$. But this means that $f_*(\pi_1(K))$ is contained in the finite-index proper subgroup $p^{-1}(\Gamma)$. But by Lemma~\ref{lem:same-finite-index-subgroups} this contradicts our hypothesis that  $\what{f}_*\colon \what{\pi_1(K)}\to \what{\pi_1(M)}$ is an isomorphism.
\end{proof}

\subsection{Proof of Theorem~\ref{thm:grothendieck-rigid} - reduction to the irreducible case}

Clearly the statement of Theorem~\ref{thm:grothendieck-rigid} is trivial if $M\cong S^1\times S^2$, thus we can henceforth assume that $M$ is irreducible.

\subsection{Proof of Theorem~\ref{thm:grothendieck-rigid} - reduction to the case that $K$ has non-empty toroidal boundary}
Let $M$ be an irreducible 3-manifold with empty or toroidal boundary and let $G\subset \pi_1(M)$ be a finitely generated subgroup  such that the induced map $\what{G}\to \what{\pi_1(M)}$ is  an isomorphism.
We need to show that $G =\pi_1(M)$. Clearly the statement is trivial if $\pi_1(M)$ is finite, thus we can assume that $M$ is aspherical. This implies that $\pi_1(M)$ is torsion-free and thus $G$ is also torsion-free.


Let $q\colon M_{G} \to M$ be the  covering of $M$ with $\pi_1(M_{G}) = G$.  By hypothesis  $M$ is irreducible. By the equivariant sphere theorem, see \cite[Theorem~3~on~p.~647]{MSY82}, the manifold 
$M_{G}$ is also irreducible. 

By Scott's core theorem~\cite{Sco73} $M_{G}$ admits a compact submanifold $K$ such that $\pi_1(K)\to \pi_1(M_{G})$ is injective and the image of $\pi_1(K)$ equals $G$. (Here the 3-manifold $K$ can a priori have any type of boundary.)  We will now identify $\pi_1(K)$ with $G$. Since $M_G$ is irreducible we can cap off in $M_G$ all boundary components of $K$ that are spheres. Thus we can assume that all boundary components of $K$ have genus at least one. 
Since $M_G$ is irreducible we see that $K$ itself is irreducible.
Since $G$ is torsion-free and non-trivial we see that $K$ is in fact aspherical.


\begin{claim}\
\bn
\item for all $i$ the map $q$ induces an isomorphism $H_i(K;\Z)\to H_i(M;\Z)$,
\item we have $\chi(K)=\chi(M)$.
\en
\end{claim}

\begin{proof}[Proof of the Claim]
Let $i\in \{0,1,2,3\}$ and let $p$ be a prime. 
We consider the following commutative diagram
\[ \xymatrix@R0.5cm@C1.1cm{ H^i(\pi_1(M);\F_p)\ar[d]\ar[r]^{q^*} & H^i(\pi_1(K);\F_p)\ar[d] \\ H^i(\what{\pi_1(M)};\F_p)\ar[r]^{\what{q}^*}&H^i(\what{\pi_1(K)};\F_p).}\]
By Theorem~\ref{thm:good} the groups $\pi_1(M)$ and $\pi_1(K)$ are good. If we apply the goodness property to $A=\F_p$ with the trivial action we obtain that in the above commutative diagrams the vertical maps are isomorphisms. By our  hypothesis on $K$ and $M$ the bottom horizontal map is an isomorphism. It follows that  the top horizontal map is also an isomorphism. Since $K$ and $M$ are aspherical we deduce that  the maps $q^*\colon H^i(M;\F_p)\to H^i(K;\F_p)$ are isomorphisms for any $i$ and any prime $p$. By the Universal Coefficient Theorem and the classification of finitely generated abelian groups this implies the first statement and evidently then also the second statement.
This concludes the proof of the claim.
\end{proof}

Now suppose that $M$ is closed. As we just showed in claim (1), this implies that $K$ is also closed 
and the map $q_*\colon H_3(K;\Z)\to H_3(M;\Z)$ is an epimorphism, i.e.\ $p\colon K\to M$ is a map of degree one. It follows from a standard argument, see e.g.\ \cite[Lemma~15.12]{He76}, that $q_*\colon \pi_1(K)\to \pi_1(M)$ is an epimorphism. The hypothesis that $\what{q_*}\colon \what{\pi_1(K)}\to \what{\pi_1(M)}$ is an isomorphism together with the fact that $\pi_1(K)$ and $\pi_1(M)$ are residually finite also implies that $q_*\colon \pi_1(K)\to \pi_1(M)$ is a monomorphism. Thus we see that $q_*\colon \pi_1(K)\to \pi_1(M)$ is an isomorphism, i.e.\ $G=\pi_1(M)$, as desired.

Now suppose that $M$ has non-trivial boundary.
By the above claim and by our hypothesis on the boundary of $M$ we have $\chi(K)=\chi(M)=0$. A standard Poincar\'e duality argument shows that $\chi(\partial K)=2\chi(K)$. Thus we see that 
$\chi(\partial K)=0$. Since $K$ has no spherical boundary components we see that all boundary components of $K$ are tori.

\subsection{Proof of Theorem~\ref{thm:grothendieck-rigid} - reduction to clean 3-manifolds with girth $\geq 3$}
In light of the previous section and of Theorem~\ref{thm:deformationtheorem} (1) it remains to prove the following proposition.

\begin{proposition}\label{prop:proof-section-2}
Let $f\colon K\to M$ be a nice map between two  irreducible 3-manifolds $K$ and $M$ that both have non-empty toroidal boundary. If $f$  induces an isomorphism of profinite completions,
then $f_*\colon \pi_1(K)\to \pi_1(M)$ is an isomorphism.
\end{proposition}

Let $f\colon K\to M$ be a nice map between two  irreducible 3-manifolds $K$ and $M$ that have non-empty toroidal boundary. Lemmas~\ref{lem:samefinitequotients}, \ref{lem:virtiso} and \ref{lem:virtiso-2} together with Propositions \ref{prop:covered-by-clean} and \ref{prop:JSJ-girth}  show that after  possibly going to a finite cover, it suffices to prove Proposition \ref{prop:proof-section-2} when $K$ and $M$ are clean manifolds and their JSJ-graphs have girth $\geq 3$. (Note that here we use the fact that the finite cover of a graph with girth $\geq 3$ has again girth $\geq 3$.) So we we can reduce the proof of Theorem~\ref{thm:grothendieck-rigid}  to the following statement.

\begin{proposition}\label{prop:proof-section-3}
Let $f\colon K\to M$ be a nice map between two  clean 3-manifolds $K$ and $M$ with non-empty toroidal boundary and whose  JSJ-graphs $\Gamma(K)$ and  $\Gamma(M)$ have girth $\geq 3$. 
If $f$  induces an isomorphism of profinite completions, then $f_*\colon \pi_1(K)\to \pi_1(M)$ is an isomorphism.
\end{proposition}

The proof of  Proposition \ref{prop:proof-section-3} will be given in the next section.

\subsection{Proof of  Proposition~\ref{prop:proof-section-3}}
Let $f\colon K\to M$ be a nice map between two  clean 3-manifolds $K$ and $M$  with non-empty toroidal boundary,  whose  JSJ-graphs $\Gamma(K)$ and  $\Gamma(M)$ have girth $\geq 3$. We denote by $f\colon \Gamma(K)\to \Gamma(M)$ the map that is induced by Lemma~\ref{lem:induced-graph map}.

\begin{lemma}\label{lem:vertices bijection} The map $f\colon \Gamma(K)\to \Gamma(M)$ induces a bijection $V(\Gamma(K)) \to V(\Gamma(M))$ on the set of vertices. 
\end{lemma}

\begin{proof}
Since the manifolds $K$ and $M$ are clean, it follows, as we had observed in Section~\ref{section:clean}, that no piece in their respective JSJ-decompositions are twisted $I$-bundles over the Klein bottle. Then the actions of $\pi_1(K)$ on the JSJ-Bass-Serre tree 
$\mathfrak{T}(K)$ and of $\pi_1(M)$ on the Bass-Serre tree $\mathfrak{T}(M)$ are $2$-acylindrical. It follows from Lemma~\ref{lem:action-is-acylindrical} that the corresponding actions of $\what{\pi_1(K)}$ and $\what{\pi_1(M)}$ on the profinite Bass-Serre trees  $\what{\mathfrak{T}}(K)$ and $\what{\mathfrak{T}}(M)$ are also 2-acylindrical. 

Let $f\colon K \to M$ be a nice map such that $\what{f_*}\colon \what{\pi_1(K)} \to \what{\pi_1(M)}$ is an isomorphism. Then $\what{\pi_1(M)}$ acts $2$-acylindrically on the profinite tree 
$\what{\mathfrak{T}}(K)$ with abelian edge stabilizers  via the isomorphism $\what{f_*}^{-1}$. Therefore for each vertex $v \in V(\Gamma(M))$ the profinite vertex group $\what{\pi_1(M_v)}$ acts acylindrically on the profinite tree $\what{\mathfrak{T}}(K)$. Since $M_v$ is not a twisted $I$-bundle over the Klein bottle, Lemma \ref{lem:fixed-vertex} implies that $\what{f_*}^{-1}(\what{\pi_1(M_v)})$ fixes a unique vertex $\wti{w} \in V(\what{\mathfrak{T}}(K))$. Let $\what{p}_K(\wti{w}) = w \in V(\Gamma(K))$ its projection in $\Gamma(K) =\what{\mathfrak{T}}(K)/ \what{\pi_1(K)}$. Then, up to conjugation, we have the following inclusions:
\[ \what{f_*}^{-1}(\what{\pi_1(M_v)}) \,\,\subset\,\, \what{\pi_1(K_w)} \,\,\subset\,\, \what{f_*}^{-1}(\what{\pi_1(M_{f(w)})}).\]
Here the first inclusion comes from Theorem~\ref{thm:profinitegraph} (b3) and the second inclusion comes from the fact that $f$ defines a map of groups.
It follows from the inclusions that $\what{\pi_1(M_v)} \subset \what{\pi_1(M_{f(w)})}$ and by Lemma~\ref{lem:fixed-vertex} we have $f(w) =v$. Since $w \in V(\Gamma(K))$ is unique, $f$ induces a bijection $V(\Gamma(K)) \to V(\Gamma(M))$.
\end{proof}

\begin{lemma}\label{lem:vertex groups}
For each $v\in V(K)$ the induced map $f_*\colon \pi_1(K_v)\to \pi_1(M_{f(v)})$ is an isomorphism.
\end{lemma}

\begin{proof}
Let $v \in V(K)$. Since $f_*\colon \pi_1(K) \to \pi_1(M)$ is injective, the restriction $f_*\colon \pi_1(K_v)$ $\to \pi_1(M_{f(v)})$ is injective. On the other hand it follows from the proof of Lemma \ref{lem:vertices bijection} that  $\what{f_*}(\what{\pi_1(K_w)}) = \what{\pi_1(M_{f(w)})}$. Hence the monomorphism $f_*\colon \pi_1(K_v)\to \pi_1(M_{f(v)})$ induces an isomorphism 
$\what{f_*}\colon \what{\pi_1(K_w)} \to \what{\pi_1(M_{f(w)})}$. The fundamental groups of  Seifert fibered 3-manifolds are LERF by \cite{Sco78} and thus Grothendieck rigid by \cite[Corollary 2.6]{LR11}.  The fundamental group of a finite volume cusped  hyperbolic 3-manifold is also Grothendieck rigid by \cite[Theorem 1.2]{LR11} (in fact these groups are also LERF  by  Wise~\cite{Wis09,Wis12a,Wis12b}, see also  \cite[Flowchart~4]{AFW15}). Hence the vertex groups $\pi_1(M_{f(v)})$ are Grothendieck rigid. Therefore $f_*\colon \pi_1(K_v)\to \pi_1(M_{f(v)})$ is an isomorphism.
\end{proof}

It follows from Lemmas~\ref{lem:epimorphism}, \ref{lem:vertices bijection} and \ref{lem:vertex groups}
that the following lemma concludes the proof of  Proposition~\ref{prop:proof-section-3}.

\begin{lemma}\label{lem:edges} The map $f\colon \Gamma(K)\to \Gamma(M)$ induces a bijection $E(\Gamma(K)) \to E(\Gamma(M))$ on the set of edges. Moreover for each edge 
$e \in E(\Gamma(K))$ the map $f_*\colon \pi_1(T_e) \to \pi_1(T_{f(e)})$ is an isomorphism.
\end{lemma}

\begin{proof}  
First we show that the map $f\colon \Gamma(K)\to \Gamma(M)$ induces an injection $E(\Gamma(K)) \to E(\Gamma(M))$ on the set of edges.
Let $e_1, e_2 \in E(\Gamma(K))$ such that $f(e_1) = f(e_2) \in E(\Gamma(M))$. 
By Lemma~\ref{lem:vertices bijection} we already know that  $f$ is bijective on the set of vertices. It thus follows from $f(e_1)=f(e_2)$ that the initial and terminal vertices of  $e_1$ and $e_2$ in $\Gamma(K)$ agree. 
It follows from this observation, together with our hypothesis that the girth of $\Gamma(K)$ is $\geq 3$, that  $e_1 = e_2$.

Now we show that $f\colon E(\Gamma(K)) \to E(\Gamma(M))$ is in fact a bijection.
We denote by $v$ the number of vertices of $\Gamma(K)$ and we denote by $e$ the number of edges of $\Gamma(K)$. Similarly we define $e'$ and $v'$ for $M$ instead of $K$. 
By the above we have $e\leq e'$. We need to show that $e=e'$. 
By Lemma~\ref{lem:vertices bijection} we have $v=v'$, which implies that $e-v\leq e'-v'$.
On the other hand Proposition~\ref{prop:2} implies that $e-v\geq e'-v'$.
Combining these inequalities we obtain that $e=e'$, as desired.

We turn to the proof of the last statement of the lemma.
For each edge $e \in E(\Gamma(K))$  the group $f_*(\pi_1(T_e))$ injects into  $\pi_1(T_{f(e)})$. Since $K$ and $M$ are clean manifolds it follows from \cite[Chapter~2.5]{AFW15} that the edge groups $\pi_1(T_e)$ and $\pi_1(T_{f(e)})$ are maximal abelian subgroups in the vertex groups $\pi_1(K_v)$ and $\pi_1(M_{f(v)})$ respectively, where $v \in V(\Gamma(K))$ is an extremity of $e$. Since by Lemma \ref{lem:vertex groups} 
the induced map $f_*\colon \pi_1(K_v)\to \pi_1(M_{f(v)})$ is an isomorphism, it follows that $f_*(\pi_1(T_e)) = \pi_1(T_{f(e)})$ and thus the induced map $f_*\colon \pi_1(T_e) \to \pi_1(T_{f(e)})$ is an isomorphism.
\end{proof}


\begin{thebibliography}{10}



\bibitem[AF13]{AF13}
M. Aschenbrenner and S. Friedl,
{\em 3-manifold groups are virtually residually $p$}, Mem. Amer. Math. Soc.
225 (2013), no. 1058.
	
\bibitem[AFW15]{AFW15}
M. Aschenbrenner, S. Friedl and H. Wilton,
{\em 3-manifold groups}, EMS Series of Lectures in Mathematics (2015)


\bibitem[BG04]{BG04}
M. Bridson and  F. Grunewald, {\em Grothendieck’s problems concerning profinite completions and representations of groups}, Ann. of Math. (2) 160 (2004), no. 1, 359--373.

\bibitem[Ca12]{Ca12}
W. Cavendish, {\em Finite-sheeted covering spaces and solenoids over 3-manifolds},
PhD thesis, Princeton University, 2012.

\bibitem[CF17]{CF17}
D. Cooper and D. Futer, {\em Ubiquitous quasi-Fuchsian surfaces in cusped hyperbolic 3-manifolds}, Preprint (2017), arXiv:1705.02890.

\bibitem[C-B13]{C-B13}   O. Cotton-Barratt, {\em Detecting ends of residually finite groups in profinite completions}, Math. Proc. Camb. Philos. Soc., 155 (2013), 379--389.

\bibitem[GAW92]{GAW92}
F. Gonz\'alez-Acu\~na and W. Whitten, 
{\em Imbeddings of three-manifold groups},
Mem. Am. Math. Soc. 99, No.474 (1992).

\bibitem[Gr70]{Gr70}
A. Grothendieck, {\em Representations lin\'eaires et compactification profinie des ´
groupes discrets}, Manuscripta Math. 2 (1970), 375--396.

\bibitem[GM17]{GM17}
D. Groves and J. Manning, {\em Quasiconvexity and Dehn filling}, Preprint (2017), arXiv:1708.07968.

\bibitem[Ha49]{Ha49}
M. Hall, {\em Coset representations in free groups}, Trans. Amer. Math. Soc. 67 (1949), 421--432.


\bibitem[HWZ13]{HWZ13} 
E. Hamilton, H. Wilton, and P. Zalesskii, {\it Separability of double cosets and conjugacy classes in 3-manifold groups}, J. Lond. Math. Soc. { 87} (2013), 269--288.

\bibitem[He76]{He76} J. Hempel, {\em $3$-Manifolds},
Ann. of Math. Studies, no. 86. Princeton University Press, Princeton, N. J., 1976.

\bibitem[He87]{He87}
J. Hempel, {\em Residual finiteness for $3$-manifolds}, Combinatorial group theory and topology (Alta, Utah, 1984), pp. 379--396, Ann. of Math. Stud., 111, Princeton Univ. Press, Princeton, NJ, 1987.


\bibitem[HZ12]{HS12} W. Herfort and P. Zalesskii, {\it Addendum: Virtually Free pro-$p$ groups whose Torsion Elements have finite Centralizer}, Preprint (2012), arXiv:0712.4244.


\bibitem[JS79]{JS79}
W. Jaco and P. Shalen. {\em Seifert fibered spaces in 3-manifolds}, Mem. Amer. Math. Soc. 21 (1979), no. 220.

\bibitem[Jo79]{Jo79}
K. Johannson, {\em Homotopy equivalences of 3-manifolds with boundaries},
Lecture Notes in Mathematics. 761. Berlin-Heidelberg-New York: Springer-Verlag (1979).



\bibitem[LR11]{LR11}
D. D. Long and A. Reid, {\em  Grothendieck's problem for 3-manifold groups}, Groups, Geometry and Dynamics, 5 (2011), 479--499.




\bibitem[MSY82]{MSY82}
W. H. Meeks, L. Simon and S. Yau, {\em Embedded minimal surfaces, exotic spheres, and
manifolds with positive Ricci curvature}, Ann. of Math. (2) 116 (1982), no. 3, 621--659.


\bibitem[NS07]{NS07} N. Nikolov and D. Segal, {\em On finitely generated profinite groups I: Strong completeness and uniform bounds}, Ann. of Math.  165  (2007), 171--236.

\bibitem[PT86]{PT86}
V. Platonov and O. Tavgen', {\em On the Grothendieck problem of profinite completions of
groups}, Dokl. Akad. Nauk SSSR 288 (1986), no. 5, 1054--1058.

%

\bibitem[Rei15]{Rei15}
A. Reid, {\em Profinite properties of discrete groups}, Proceedings
of Groups St Andrews 2013, L.M.S. Lecture Note Series  242,
Cambridge Univ. Press (2015), 73--104.

\bibitem[Rib17]{Rib17} L. Ribes, {\em Profinite Graphs and Group}, Ergebnisse der Mathematik und ihrer Grenzgebiete 3. Folge. A Series of Modern Surveys in Mathematics  66. Springer-Verlag, Berlin, 2017.

\bibitem[RZ96]{RZ96} L. Ribes and P. Zalesskii, {\em Conjugacy separability of amalgamated
free products of groups}, J. Algebra (3) 179  (1996), 751--774.

\bibitem[RZ00]{RZ00} L. Ribes and P. Zalesskii, {\em Pro-p trees and applications}, in New horizons in pro-p groups, 75--119. Springer, 2000.


\bibitem[RZ10]{RZ10}
L. Ribes and P. Zalesskii, {\em Profinite groups. Second edition}, Ergebnisse der Mathematik und ihrer Grenzgebiete. 3. Folge. A Series of Modern Surveys in Mathematics 40. Springer-Verlag, Berlin, 2010.


\bibitem[Sco73]{Sco73}
P. Scott, {\em Compact submanifolds of 3-manifolds}, J. London Math. Soc. (2) 7
(1973), 246--250.

\bibitem[Sco78]{Sco78}
P. Scott, {\em  Subgroups of surface groups are almost geometric},
J. London Math. Soc. { 17} (1978), 555--565; Correction ibid. { 32}
(1985), 217--220.


\bibitem[Se80]{Se80}
J. P. Serre, {\em  
Trees}, Springer-Verlag (1980).

\bibitem[Se97]{Se97}
J. P. Serre, {\em  Galois Cohomology}, Springer, Berlin, 1997.





\bibitem[Wa68]{Wa68}
F. Waldhausen, {\em 
On irreducible 3-manifolds which are sufficiently large}, Ann. of Math.
(2) 87 (1968), 56--88.


\bibitem[Wil17a]{Wil17a} G. Wilkes, {\em Profinite rigidity of graph manifolds and JSJ-decompositions of 3-manifolds}, Preprint (2017), arXiv:1605.08244v2.

\bibitem[Wil17b]{Wil17b} G. Wilkes, {\em  Relative cohomology theory for profinite groups}, Preprint (2017), arXiv:1710.00730.



\bibitem[WZ10]{WZ10}
H. Wilton and P. Zalesskii, {\em Profinite properties of graph manifolds}, Geom. Dedicata
147 (2010), 29--45.

\bibitem[WZ17a]{WZ17a}
H. Wilton and P. Zalesskii, {\em Distinguishing geometries using finite quotients}, Geom. Topol., { 21} (2017), 345--384.


\bibitem[WZ17b]{WZ17b}
H. Wilton and P. Zalesskii, {\em Profinite detection of 3-manifold decompositions}, arXiv:1703.03701.

\bibitem[Wis09]{Wis09}
D. Wise, {\em The structure of groups with a quasi-convex hierarchy}, Electronic Res. Ann. Math. Sci 16 (2009), 44--55.

\bibitem[Wis12a]{Wis12a}
D. Wise, {\em The structure of groups with a quasi-convex hierarchy}, 189 pages, preprint (2012),
downloaded on October 29, 2012 from \\
\texttt{http://www.math.mcgill.ca/wise/papers.html}

\bibitem[Wis12b]{Wis12b}
D. Wise, {\em From riches to RAAGs: $3$-manifolds, right--angled Artin groups, and cubical geometry}, CBMS Regional Conference Series in Mathematics, 

\bibitem[ZM89]{ZM89} P. Zalesskii and  O.V. Mel'nikov, {\em Subgroups of profinite groups acting on trees}, Math.
USSR-Sb. 63 (1989), 405--424.

\bibitem[ZM90]{ZM90} P. Zalesskii and  O.V. Mel'nikov, {\em Fundamental groups of graphs of profinite groups}, Leningrad Math J. 1 (4) (1990), 921--940.


\end{thebibliography}
\end{document}